\documentclass[12pt,reqno]{amsart}
\usepackage[mathscr]{eucal}

\def\ri{\mathrm{i}}
\def\C{{\mathbb C}}
\def\R{{\mathbb R}}
\def\fR{{\mathfrak R}}
\def\CP{{\C}P}
\def\RP{{\R}P}
\def\gam{\mbox{\raisebox{.45ex}{$\gamma$}}}
\def\bgam{\mbox{\raisebox{.45ex}{$\boldsymbol{\gamma}$}}}
\renewcommand{\Re}{\operatorname{Re}}

\def\scrN{\EuScript N}
\def\calB{\mathcal B}
\def\calR{\mathcal R}
\def\calQ{\mathcal Q}
\def\calP{\mathcal P}

\def\calM{\mathcal M}
\def\calS{\mathcal S}
\def\calE{\mathcal E}
\def\calO{\mathcal O}
\def\Schwa{\mathfrak S} 
\def\tGamma{\widetilde{\Gamma}}
\def\impart{\operatorname{\sf Im}}
\def\realpart{\operatorname{\sf Re}}
\def\vec#1{\mathbf{#1}}

\def\di{\partial}

\theoremstyle{definition}
\newtheorem{thm}{Theorem}
\newtheorem{prop}[thm]{Proposition}
\newtheorem{lemma}[thm]{Lemma}
\newtheorem{cor}{Corollary}
\newtheorem*{defn}{Definition}
\newtheorem{nex}{Example} 
\numberwithin{nex}{section}
\newtheorem{rem}{Remark}

\usepackage{amsaddr}

\begin{document}
\title[Geometric Flows for Curves in $S^3$]{Integrable Geometric Flows for \\ Curves in Pseudoconformal $S^3$}
\author{Annalisa Calini \& Thomas Ivey}
\address{Department of Mathematics, College of Charleston, Charleston, SC, USA}
\email{calinia@cofc.edu, iveyt@cofc.edu}
\begin{abstract}
{{We consider evolution equations for curves in the 3-dimensional sphere $S^3$ that are invariant under the group $SU(2,1)$ of pseudoconformal transformations, which preserves the standard contact structure on the sphere. In particular, we investigate how invariant evolutions of Legendrian and transverse curves induce well-known integrable systems and hierarchies at the level of their geometric invariants.}}
\end{abstract}
\maketitle

\section{Introduction}
In this article we are motivated by our interest in evolution equations for parametrized curves in homogeneous spaces, and more specifically evolution equations that are invariant under the Lie group $G$ of congruences
of the space in question.  This article will focus on the specific case of the non-compact Lie group $SU(2,1)$ acting on the 3-dimensional sphere.

In general, we refer to geometric evolution equations for curves as `flows', and they will in turn induce systems of evolution equations for the set of differential invariants of the curves in that geometry.  We are interested in identifying flows such that the induced evolution equations for the invariants form a completely integrable system of partial differential equations; we will refer to these evolution equations for curves as {\em integrable flows}.  Somewhat surprisingly, the most well-studied examples of such flows do not involve the KdV equation.

\begin{nex} The vortex filament flow is an evolution equation for parametrized curves in Euclidean space $\R^3$ introduced by Da Rios \cite{daRios} as a model of the self-induced motion of a thin vortex filament in an incompressible fluid.  It takes the form
$$\dfrac{\di \gam}{\di t} = \dfrac{\di \gam}{\di x} \times \dfrac{\di^2 \gam}{\di x^2},$$
where $\times$ is the Euclidean cross product and $x$ is the parameter along the curve.  Since the speed
$|\di \gam/\di x|$ is preserved pointwise under this flow, one usually assumes that $x$ is a unit-speed parameter, whereupon the flow becomes
\begin{equation}\label{VFE}
\dfrac{\di\gam}{\di t} = \kappa B,
\end{equation}
where $B$ is the binormal vector (part of the classical Frenet moving frame) and $\kappa$ is the curvature of $\gam$.  The other fundamental differential invariant of the curve is the torsion $\tau$, and if one packages the curvature and torsion into one complex-valued function
\begin{equation}\label{qform} q(x,t) = \kappa e^{\int \tau\,dx},
\end{equation}
known as the \emph{Hasimoto transformation} \cite{Hasimoto},
then $q$ satisfies\footnote{In order to obtain the NLS evolution for $q$, the antiderivative in the exponent in \eqref{qform} must be
chosen so that the constant of integration evolves in a particular way; see \cite{IL18} for details.} %
$q_t =\ri (q_{xx} + \tfrac12 |q|^2q)$,
the focusing cubic nonlinear Schr\"odinger equation (NLS).
The integrable structure of this PDE can be exploited to study geometric and topological properties \cite{CI2005,CI2007} as well as stability
\cite{CI2011, CKL2011} of the corresponding solutions of \eqref{VFE}.
\end{nex}

\begin{nex}\label{mKdVflows} The vortex filament flow is part of an infinite sequence of flows for arclength-parametrized curves in Euclidean space, each of which corresponds, under the Hasimoto transformation \eqref{qform}, to a member of NLS hierarchy.  For example, the next flow after \eqref{VFE} in the
sequence is
\begin{equation}\label{VFE3}
\gam_t = \tfrac12 \kappa^2 T + \kappa_x N + \kappa \tau B,
\end{equation}
(where again $T,N,B$ are the members of the Frenet frame and $x$ is arclength), and this corresponds to
$$q_t = q_{xxx} + \tfrac32 |q|^2 q_x.$$
It is not hard to show that, if a curve is planar (i.e., with $\tau=0$ identically) then this feature is preserved by the flow \eqref{VFE3}, and then
$q$ equals the real-valued curvature $\kappa$, which thus satisfies the mKdV equation.  In fact, every second flow in the VFE hierarchy restricts to planar curves, and these induce the mKdV hierarchy at the level of curvature \cite{LP98}.
\end{nex}

It has long been known that the KdV equation and its integrable generalizations are associated to the special linear groups $SL(n,\R)$ \cite{DSok84}, and in fact one of the first geometric realizations of the KdV equation arose from a flow in $\R^2$,
invariant under $SL(2,\R)$ acting in the standard way, for curves satisfying a genericity assumption:

\begin{nex}\label{pinkex} A parametrized regular curve $\gam:I \to \R^2$ is called {\em star-shaped} if $\det(\gam, \gam_x)$ is non-vanishing for all $x\in I \subseteq \R$.  Normalizing the parametrization so that
$\det(\gam,\gam_x)=1$ and differentiating, we get $\gam_{xx} =-p(x) \gam$ for a curvature function $p$.  The flow $\gam_t = \tfrac12 p_x \gam - p \gam_x$ was shown by Pinkall \cite{P95} to be Hamiltonian for the total curvature (with respect to a natural symplectic structure), and to induce the KdV equation for the curvature $p$.  Similarly, there are flows for star-shaped curves that induce each evolution equation
in the KdV hierarchy \cite{CQ02}.
\end{nex}

\begin{nex}\label{centroaffineR3} In the higher-dimensional generalization of the previous example, the space $\R^n$ is acted on by $SL(n,\R)$ via the standard linear representation, and is known as {\em centroaffine} space.  We can consider invariant flows for generic parametrized curves in this space and the induced evolutions of their differential invariants (known as {\em Wilczynski invariants}).

For example, mappings $\gam:I \to \R^3$ are called {\em starlike} if $\det(\gam,\gam', \gam'')$ is nonvanishing, and we when normalize the parameter so that this is identically equal to one, setting $\gam''' =p_0 \gam + p_1\gam'$ yields two scalar differential invariants.  In \cite{CIM} we show that there are flows for curves in centroaffine $\R^3$ that induce, at the level of the invariants, every PDE system in the Boussinesq hierarchy.  We also show that, within this sequence of flows there is a subsequence that preserves the condition $p_0 = \tfrac12 p_1'$, indicating that the image of $\gam$ lies on a fixed cone in $\R^3$, and these flows induce the Kaup-Kuperschmidt hierarchy for the curvature $p_1$.  (We think of this as roughly analogous to the realization, in Example \ref{mKdVflows}, of the mKdV hierarchy by planarity-preserving flows within the VFE hierarchy.)
\end{nex}

\subsection*{The pseudoconformal 3-sphere}
In this article, we investigate geometric flows for curves in another 3-dimensional geometry defined by the action of an 8-dimensional matrix group,
in this case the action of $SU(2,1)$ on the 3-dimensional sphere $S^3$.  To see how this action is defined, on $\C^3$ fix an indefinite Hermitian form
\begin{equation}
\label{sesqform}
\langle \vec{z}, \vec{w}\rangle = - z_0 \overline{w_0} + z_1 \overline{w_1} + z_2 \overline{w_2}.
\end{equation}
Taking $SL(3,\C)$ to act on $\C^3$ in the standard way, we define $SU(2,1)$ as the subgroup that preserves this form.
Explicitly, if we let $J=\text{diag}(-1, 1,1)$, so that $\langle \vec{z} , \vec{w} \rangle=\vec{z}^t J \overline{\vec{w}}$, then
$$
SU(2,1)=\{G \in SL(3, \C) \, | \, G^TJ\overline{G}=J\}.
$$

The set of nonzero null vectors for \eqref{sesqform} is a cone $\scrN \subset \C^3$, which is preserved by multiplication by complex scalars.

\begin{lemma}\label{identifysphere} The image of $\scrN$ under complex projectivization $\pi: \C^3 \backslash \{0\} \to \CP^2$ is diffeomorphic to the unit sphere $S^3 \subset \C^2$.
\end{lemma}
\begin{proof} Since any nonzero vector $\langle \vec{z},\vec{z}\rangle = 0$ must have component $z_0 \ne 0$, the image $\pi(\scrN)$ lies entirely within the domain of affine coordinates
$Z_1 = z_1/z_0, Z_2 = z_2/z_0$.  Dividing the defining equation $|z_1|^2 + |z_2|^2=|z_0|^2$ by $|z_0|^2$ shows that the image is the set of points in $\C^2$ satisfying $|Z_1|^2 + |Z_2|^2=1$.
\end{proof}

From now on we will identify $S^3$ with the projectivization of $\scrN$ without further comment.
The induced action of $SU(2,1)$ on $S^3$ is what we will call the group of {\em pseudoconformal transformations} of the sphere; explicitly,  $G \cdot \pi(\vec{z}):=\pi(G\vec{z}).$

Our work on curves in this geometry was inspired by a paper by Musso \cite{musso}, where it is noted that the pseudoconformal group preserves a contact structure on $S^3$ (which will be defined below).  Focusing on Legendrian curves, Musso constructed an adapted moving frame, defined a pseudoconformally-invariant arclength and curvature, and obtained geodesic and elastic Legendrian curves.

In this paper we concentrate both on Legendrian curves and curves that are everywhere transverse to the contact distribution.  Our results for Legendrian parametrized curves echo our previous work mentioned in
Example \ref{centroaffineR3},  in that we define a sequence of flows which realize the Boussinesq hierarchy, and identify a subsequence which preserve a geometrically natural condition (arclength parametrization)
and induce the Kaup-Kuperschmidt hierarchy (see Theorems \ref{bouthm} and \ref{KKthm} below).  For transverse curves we define a similar set of differential invariants, and identify several flows inducing
integrable evolution equations for these invariants (see Example \ref{lpreserving}), as well as some flows that induce evolution equations for the curvatures for which there is evidence of integrability, but not definitive proof at present (see Examples \ref{mikex} and \ref{sinkex}).

\subsection*{Outline}
In more detail, we now summarize how the rest of the paper is laid out.  In \S2 we define the $SU(2,1)$-invariant contact structure on $S^3$ and introduce moving frames and invariant arclength for regular parametrized Legendrian curves.  In \S3 we determine the geometric flows that preserve the Legendrian condition and how the differential invariants evolve.
In \S4 we turn to transverse curves, developing moving frames and differential invariants, geometric flows and the induced evolutions for the invariants; we also detail several specializations that lead to known integrable evolutions for these invariants.  In the final section we briefly discuss some open questions.

\section{Pseudoconformal Geometry of Curves in $S^3$}

\begin{defn}
Given a parametrized curve $\gam:I \to S^3$, $\gam$  is {\em Legendrian} if it has a lift $\Gamma:I \to \scrN$
satisfying
\begin{equation}\label{contact-cond}
\langle \Gamma_x, \Gamma \rangle = 0.
\end{equation}
Note that this constitutes a single real condition, since it follows automatically
from $\Gamma$ taking value in the null cone $\scrN$ that $\realpart\langle \Gamma_x, \Gamma\rangle = 0$.
\end{defn}

We remark that the pseudoconformal group preserves this condition, and that the set of tangent directions to Legendrian curves forms a well-defined contact distribution on the sphere. Furthermore this contact structure is equivalent to the `standard' contact structure on $S^3$, when we realize the latter as the unit sphere in $\C^2$ (as explained above).

We will refer to  regular parametrized Legendrian curves as {\em L-curves}.  In what follows, we will construct  pseudoconformally-invariant adapted moving frames for L-curves.

\begin{defn} Given linearly independent vectors $e_0,e_1,e_2$ in $\C^3$, we will call $(e_0,e_1,e_2)$ a {\em null frame} for $\C^3$ if
\begin{equation}\label{innerproductconds}
\begin{split}
\langle e_0, e_0 \rangle = \langle e_0, e_1 \rangle = \langle e_2, e_1 \rangle = \langle e_2, e_2 \rangle =0, \\
\langle e_2, e_0 \rangle = \ri, \qquad \langle e_0, e_2 \rangle = -\ri, \qquad \langle e_1, e_1 \rangle = 1,
\end{split}
\end{equation}
along with
\begin{equation}\label{determinantcond}
\det(e_0, e_1, e_2) = 1.
\end{equation}

The set of null frames is acted on freely and transitively by $SU(2,1)$.  In fact, this set is isomorphic to $SU(2,1)$, since a $3\times 3$ matrix
$G$ lies in $SU(2,1)$ if and only if its columns $e_0,e_1,e_2$ satisfy the above conditions.


A {\em framing} for $\gam:I \to S^3$ is a null frame field such that $\pi \circ e_0 = \gam$, i.e., for each $x$ the first null vector $e_0(x)$ of the frame field always points along the complex line in the null cone spanned by $\gam(x)$.
\end{defn}

We begin by showing that adapting the frame field allows us to define a trio of real-valued relative invariants.  The following choice of framing essentially corresponds to the third-order reduction of the frame bundle defined by Musso \cite{musso}.

\begin{prop}
\label{firstLadapt} Any  L-curve $\gam(x)$ has a framing $(\Gamma, T, N)$ that satisfies:
\begin{equation}\label{firstLfrenet}
\Gamma_x = \nu T, \qquad T_x = \ri \nu  N + k \Gamma, \qquad N_x = \ell \Gamma - \ri k T
\end{equation}
for real-valued functions $\nu,k,\ell$, with $\nu$ positive.
\end{prop}
We will refer to this an {\em adapted framing} for an L-curve.  Note that the components of the framing may be expressed in terms of the lift $\Gamma$ and its first two derivatives.
\begin{proof}
Let $(e_0,e_1,e_2)$ be any framing of $\gam$, and let
\begin{equation}\label{initiale0prime}
e_0' = a e_0 + b e_1
\end{equation}
for some complex-valued functions $a,b$ of $x$.  (Note that the prime here denotes differentiation with respect to $x$,
and Legendrian condition $\langle e_0', e_0\rangle =0$ implies that $e_0'$ has no $e_2$-component.)  By regularity,
$b$ is nonvanishing.

The permissible changes of  frame that keep $\pi(e_0)$ unchanged are
\begin{equation}\label{CoF}
\begin{aligned}
\tilde e_0 &= \lambda e_0 \\
\tilde e_1 &= \dfrac{\overline{\lambda}}{\lambda}  \left( e_1 +\mu e_0 \right) \\
\tilde e_2 &= {\overline{\lambda}}^{-1} \left(e_2 -\ri \overline{\mu}e_1+ (\alpha - \tfrac12 \ri |\mu|^2 \right) e_0),
\end{aligned}
\end{equation}
where $\lambda,\mu$ are complex, with $\lambda \ne 0$, and $\alpha$ is real.  By using  $\mu = a/b$, $\lambda=1$ and $\alpha=0$,
we can define a new frame such that $\tilde e_0'$ has no $\tilde{e}_0$-component.  Changing to this frame (and dropping the tildes), we now  differentiate the relations
$\langle e_1, e_0 \rangle = 0$, $\langle e_2, e_0 \rangle =\ri$ and $\langle e_2, e_1 \rangle = 0$ to find that our current frame satisfies differential equations of the form
$$e_0' = b e_1, \quad e_1' = \ri b e_2 + k e_0, \quad e_2' = \ell e_0 -\ri k e_1$$
for some functions $k$, $\ell$ of $x$.  Furthermore, by differentiating $\langle e_2, e_2 \rangle=0$ we see that $\ell$ is real-valued, but the function $k$ is in general complex-valued.

We are still allowed to make changes of the form \eqref{CoF} with $\mu=0$.   If we also set $\alpha=0$, then under these
changes the coefficient $b$ is multiplied by $\lambda^2 /\overline{\lambda}$.  Thus, we arrange that $b$ is positive.
Finally, by adding a real multiple of $e_0$ to $e_2$ we may arrange that $k$ is real-valued.  Now we let $\Gamma=e_0$,
$T=e_1$ and $N=e_2$ for our final choice of frame.
\end{proof}

We remark that the framing in Proposition~\ref{firstLadapt} is unique up to scaling $\Gamma \mapsto \lambda \Gamma$, $N \mapsto \lambda^{-1} N$ for $\lambda\in \R$, and multiplication of each vector by a cube root of unity. It follows that  the  integral  $\int (\nu^2 \ell)^{1/3} dx$ is well-defined.  Moreover, there is a  well-defined \emph{normal indicatrix} which is the parametrized curve in $S^3$ given by $\pi \circ N$.
Following Musso, we say that $\gam(x_0)$ is a {\em sextactic point} of $\gam$ if $\pi\circ N$ is tangent to the contact plane at $x=x_0$, i.e., $\ell(x_0)=0$. For curves that are free of sextactic points,  $\int (\nu^2 \ell)^{1/3} dx$ is interpreted as the {\em pseudoconformal arclength} of the L-curve. Note that if $\Gamma$ is an adapted lift as in Proposition~\ref{firstLadapt}, then
\[
\frac{\det(\Gamma_x, \Gamma_{xx}, \Gamma_{xxx})}{\det(\Gamma, \Gamma_x, \Gamma_{xx})}=k_x\nu-k\nu_x+ \ri \nu^2 \ell.
\]
Thus the arclength integrand can be expressed in terms of a ratio of determinants, and is therefore invariant under the action of the larger group $GL(3, \C)$.
(Differential invariants of curves with respect to this group were worked out in \cite{O10}.)

\begin{cor}\label{Lnormalize}
Any $L$-curve has an adapted framing for which $\nu=1$ identically, and this is unique up to simultaneously multiplying each frame vector by a cube root of unity.
We will refer to this as a {\em normalized} framing.
\end{cor}
\begin{proof} As in the proof of Prop. \ref{firstLadapt}, suppose that an initial framing satisfies \eqref{initiale0prime}.  Then we change to a framing satisfying $\tilde e_0' = \tilde e_1$ by making a change \eqref{CoF} such that $\overline{\lambda}/\lambda^2 = b$ and $b\mu - a =\lambda'/\lambda$.  We can further adjust the frame only by using the scale factor $\lambda$, and this must satisfy $\lambda^2 /\overline{\lambda}=1$
in order to preserve $b=1$, so $\lambda^3=1$.
\end{proof}


\section{Deformations of L-curves}
In this section we study deformations for lifted L-curves that lead to integrable evolution equations for the geometric invariants $k$ and $\ell$.
We first consider deformations that preserve the contact condition \eqref{contact-cond} and the normalization $\nu=1$.

\begin{prop}
\label{Lframevar}
Let $\gamma(x,t)$ be a smooth variation of L-curves, and let $(\Gamma, T, N)$ be a smoothly-varying adapted framing.  If we write
\begin{equation}\label{resolveGt}
\Gamma_t = f \Gamma + g T + h N
\end{equation}
then necessarily $h$ is real-valued and $h_x = 2\nu \impart g$.
If furthermore $(\Gamma, T, N)$ is a normalized framing for all $t$ then
\begin{equation}
\label{frelation}
f = -\realpart(g_x) + \tfrac13 \ri (k h - \tfrac12 h_{xx}).
\end{equation}
\end{prop}
\begin{proof}
Because $\Gamma$ takes value in the null cone then differentiating $\langle \Gamma, \Gamma\rangle = 0$ gives
\begin{equation}\label{hreal}
0 = \realpart \langle \Gamma_t, \Gamma \rangle = \realpart \langle h N, \Gamma\rangle = \realpart ( \ri h  ).
\end{equation}
Differentiating \eqref{contact-cond} with respect to $t$ gives
$$0 = \langle \Gamma_{tx}, \Gamma \rangle +\langle \Gamma_x, \Gamma_t\rangle
= \ri h_x +\nu(\overline{g}-g).$$


For the second assertion, we will compute the evolution of the frame vectors, regarded as columns in an $SU(2,1)$-valued matrix
$F(x,t)$.  For example, the Frenet-type equations of Proposition \ref{firstLadapt} with the normalization $\nu=1$ may be expressed as
\begin{equation}\label{ULcurve}
\dfrac{\partial F}{\partial x}  = F U \quad \text{for} \quad U = \begin{pmatrix} 0 & k & \ell \\ 1 & 0 & -\ri k \\ 0 & \ri & 0 \end{pmatrix}.
\end{equation}
We also have $\dfrac{\partial F}{\partial t} = F V$ where $V$ takes value in the Lie algebra $su(2,1)$:
$$V = \begin{pmatrix} f  & z & j \\ g & \overline{f}-f & -\ri \overline{z} \\ 	h & \ri \overline{g} & -\overline{f} \end{pmatrix},$$
with $h$ and $j$ real-valued.  The compatibility condition of the two equations is $H=0$ where
$$H : = U_t -V_x - [U,V].$$
For example, the bottom left entry is $-h_x + \ri(\overline{g}-g)$, which we already deduced was zero from the contact condition. Setting the $(2,1)$-entry of $H$ equal to zero and solving for $f$ gives~\eqref{frelation}.

\end{proof}
\begin{prop}\label{Levolutions}
Let $\gam(x,t)$ be a smooth variation of L-curves and let $(\Gamma, T, N)$ be a smoothly-varying normalized framing satisfying \eqref{resolveGt}.  Let $a=\Re g$, so that
\begin{equation}\label{Lunitflow}
\Gamma_t =\left(-a_x + \tfrac13 \ri (k h - \tfrac12 h_{xx}) \right)\Gamma + (a+ \tfrac12 \ri  h_x) T + h N
\end{equation}
by Prop. \ref{Lframevar}.  Then the geometric invariants $k, \ell$ evolve by
\begin{subequations}\label{Lklflow}
\begin{align}
k_t & =a k_x   + 2k a_x +h \ell_x + \tfrac32 \ell h_x - a_{xxx} , \label{Lkflow}\\
\ell_t &= a\ell_x + 3 a_x \ell + \tfrac13 h k_{xxx}  + \tfrac32 h_x k_{xx} +\tfrac52 h_{xx} k_x -\tfrac83 h k k_x  -\tfrac83 h_x k^2 +\tfrac53 h_{xxx} k -\tfrac16 h^{(5)} \label{Lelflow}
\end{align}
\end{subequations}
where the superscript $(5)$ indicates the fifth derivative with respect to $x$.
\end{prop}
\begin{proof}
We continue the calculation from the proof of Proposition \ref{Lframevar}. Setting the $(1,1)$-entry of $H$ equal to zero
lets us solve for
$$z = ak +h \ell - a_{xx} +  \tfrac16 \ri \left(  2 hk_x + 5 k h_x -h_{xxx}\right).$$
Substituting this into the superdiagonal of $H$ and gives
$$
j =a \ell -\tfrac16 h^{(4)} + \tfrac43 k h_{xx} + \tfrac76 h_x k_x + h (\tfrac13 k_{xx} - k^2).
$$
and the evolution \eqref{Lkflow} for $k$.  Finally, the evolution for $\ell$ is obtained by substituting the expressions for $j$ and $z$ into the $(1,3)$-entry of $H$.
\end{proof}
By choosing $a$ and $h$ be local functions of the invariants $k$ and $\ell$ (i.e., expressed in terms of $k,\ell$ and finitely many of their derivatives) we obtain geometric flows for L-curves. In what follows we will focus on flows that induce integrable evolution equations for invariants $k$ and $\ell$. We present a few examples.

\begin{nex}[Translation Flow]\label{trans}
By choosing $h=0$ and $a$ to be a constant we obtain
$$
k_t=ak_x, \qquad \ell_t=a\ell_x.
$$
\end{nex}

\begin{nex}[Boussinesq System]\label{bouex}
Choosing $a=0$ and  $h=-1$, we get
\begin{equation}
\label{bsnq}
\begin{aligned}
k_t & =-\ell_x, \\
\ell_t &=  \tfrac13  (-k_{xxx} + 8 k k_x),
\end{aligned}
\end{equation}
which is equivalent to the Boussinesq system. (For instance, to match Example 7.28 in \cite{Obook}, set $u=-k$ and $v=\ell$.)
\end{nex}

\begin{nex}[KdV Reduction]\label{KdVex}
Choosing $h=0$ and $a=-k$, we obtain
\begin{align*}
k_t &= k_{xxx} - 3 k k_x,\\
\ell_t &= -k \ell_x -3 k_x \ell.
\end{align*}
Thus, $k$ evolves by the KdV equation while $\ell$ satisfies a linear homogeneous equation.  In particular, this flow preserves
sextactic curves (i.e., those for which $\ell$ is identically zero).  We will discuss this special case below.
\end{nex}

\begin{nex}[Kaup-Kuperschmidt Reduction]\label{KKR}
Choosing $h$ to be a constant $\lambda$ and $a=4 k^2-k_{xx}$, we obtain
\begin{align*}
k_t & =k^{(5)} -10 k k_{xxx} - 25 k_x k_{xx} + 20 k^2 k_x+\lambda \ell_x, \\
\ell_t &=  a\ell_x +3 \left(\ell - \frac{\lambda}{9}\right) a_x,
\end{align*}
which clearly preserve the condition $\ell=\lambda/9$. In that case we obtain the reduction
\begin{equation}
\label{KK}
k_t = k^{(5)} -10 k k_{xxx} - 25 k_x k_{xx} + 20 k^2 k_x,
\end{equation}
which is equivalent to the Kaup-Kuperschmidt (KK) equation after the change of variable $u=-2k$ (see Example 2.20 in \cite{JPlist}).
\end{nex}

We now observe that there are geometric flows of the form \eqref{Lunitflow} such that the invariant evolutions \eqref{Lklflow}
realize any equation in the Boussinesq hierarchy.  Recall that this is a double hierarchy which is defined recursively as follows \cite{JPlist}:
$$
\begin{bmatrix} u \\ v \end{bmatrix}_t = F_n[u,v], \text{ where } F_{n+2} = \calR F_n,
$$
where the `seeds' of the hierarchy are given by
$$F_0 = \begin{bmatrix} u_x \\ v_x \end{bmatrix}, \qquad F_1 = \begin{bmatrix} v_x \\ \tfrac13 u_{xxx} +\tfrac83 u u_x \end{bmatrix}$$
and the recursion operator\footnote{In what follows, when symbols denoting differential operators are adjacent the product should be understood as composition.
An exception occurs when an operator $\calO$ occurs next to the symbol $f$ denoting a function or vector of functions; then $\calO f$ denotes the application of $\calO$ to $f$, while
$\calO \circ f$ denotes the composition of $\calO$ with the operator that multiplies by $f$.}
is $\calR =\calP \calQ^{-1}$ for $\calQ = \left(\begin{smallmatrix} 0 & D \\ D & 0 \end{smallmatrix}\right)$, with $D=d/dx$, and
$$\calP = \left( \begin{array}{ll}D^3 + 2u D + u_x & 3v D + 2v_x \\ 3 v D + v_x  & \tfrac13 D^5 + \tfrac53(u D^3 + D^3 \circ u) - ( u_{xx} D + D \circ u_{xx})
+ \tfrac{16}3 u D \circ u \end{array}\right).
$$
The hierarchy may be equivalently defined by writing $F_n = \calP G_n$ and defining the double sequence of
`cosymmetries' $G_n$ by
$$G_{n+2} = \calQ^{-1} \calP G_n, \text{ where } G_0 = \begin{bmatrix} 1 \\ 0 \end{bmatrix}, \quad G_1 = \begin{bmatrix} 0 \\ \tfrac12 \end{bmatrix}.$$

\begin{thm}\label{bouthm} With the change of variables $u=-k$ and $v=\ell$, choosing $[a,-\tfrac12 h]^T = G_n$ makes system \eqref{Lklflow} agree with
the $n$th flow of the Boussinesq hierarchy.
\end{thm}
\begin{proof} We rewrite \eqref{Lklflow} in the form
$$
\begin{bmatrix} k \\ \ell \end{bmatrix}_t = \begin{pmatrix} k_x + 2k D - D^3 & -3\ell D - 2 \ell_x \\
 3 \ell D + \ell_x & \calB
\end{pmatrix} \begin{bmatrix} a \\ -\tfrac12 h\end{bmatrix}
$$
where
$$\calB = \tfrac13 D^5 - \tfrac{10}3 k D^3 - 5 k_x D^2 - (3k_{xx} - \tfrac{16}3k^2 ) D - \tfrac23 k_{xxx} + \tfrac{16}3 k k_x.$$
This is a skew-adjoint operator, as can readily be seen when it is rewritten in the form
$$\calB = \tfrac13 D^5 -\tfrac53 (k D^3 + D^3 \circ k) + k_{xx} D + D \circ k_{xx} +\tfrac{16}3 k D \circ k,$$
This clearly agrees with the bottom right entry in $\calP$ under the change $u=-k$.
\end{proof}
\subsection{Sextactic L-curves and the KdV hierarchy}

If an L-curve is sextactic, i.e. $\ell$ vanishes for all $x$, then the flow for the lift in Example~\ref{KdVex} takes the particularly simple form
\begin{equation}
\label{Pinkall}
\Gamma_t = k_x \Gamma-  k \Gamma_x.
\end{equation}
More generally, whenever $h=0$ the flow for $\Gamma$ preserves the condition $\ell=0$  and takes the form
\begin{equation}
\label{hzeroflow}
\Gamma_t =- a_x \Gamma + a \Gamma_x.
\end{equation}
In terms of affine coordinates $\gamma_j = \Gamma_j/\Gamma_0$ for $j=1,2$, this corresponds to
\begin{equation}\label{adownflow}\dfrac{\partial \gamma_j}{\partial t} = a \dfrac{\partial \gamma_j}{\partial x}
\end{equation}
which of course preserves the condition that image of $\gamma$ lies on $S^3$, i.e. $|\gamma_1|^2+|\gamma_2|^2=1$.

When  $\Gamma$ is sextactic,  by a rigid motion in $SU(2,1)$ we may assume that $\gamma_1, \gamma_2$ are both real (see Prop. 2.4 in \cite{musso}), so that the image of $\gamma$ lies on a great circle in $S^3$. Moreover, since $a$ is real the evolution \eqref{adownflow} also preserves this condition.

We will take advantage of this special form to compute the pseudoconformal curvature of $\gam$.  After constructing a sequence geometric flows for sextactic curves that realize the KdV hierarchy,
we will relate these to the flows for maps into centroaffine $\R^2$ which also realize this hierarchy.

Using the procedure described in the proof of Proposition~\ref{firstLadapt}, we first introduce a null framing $e_0=\omega (-1,\cos \phi, \sin \phi)$, $e_1=\omega (0, -\sin \phi, \cos \phi)$, $e_2=\tfrac12\ri\omega (-1,\cos \phi, \sin\phi)$, where $\phi$ is a real function of $x$ and $\omega=e^{\ri \pi/6}$.  We modify this to obtain a normalized framing
\[
\Gamma=\frac{1}{\phi_x}e_0, \qquad T=e_1-\frac{\phi_{xx}}{(\phi_x)^2} e_0,  \qquad N=\phi_x e_2+\ri \frac{\phi_{xx}}{\phi_x} e_1 -\frac{\ri}{2}\frac{(\phi_{xx})^2}{(\phi_x)^3} e_0.
\]
This frame satisfies \eqref{firstLfrenet} with $\nu=1$, $\ell=0$, and the curvature is given by
$k=- \left[\Schwa(\phi) + \frac{1}{2} (\phi_x)^2\right]$ where
$$\Schwa(\phi) :=\dfrac{\phi_{xxx}}{\phi_x} - \dfrac{3\phi_{xx}^2}{2\phi_x^2}$$
denotes the Schwarzian derivative.

\begin{thm}\label{sexthm}  Let $\mathcal{E}=D^3_x -(kD_x + D_x k)$ be the second symplectic operator of the KdV equation, and $F_{j+1}[k]=\mathcal{E}\circ D_x^{-1} F_j[k]$, $F_0=k_x$ be the recursive scheme for the KdV hierarchy of vector fields.  Then, equation~\eqref{hzeroflow} with $a=-D_x^{-1} F_j[k]$ is the geometric realization of the $j$-th KdV flow as an evolution on sextactic Legendrian curves.
\end{thm}
\begin{proof}
We rewrite~\eqref{Lkflow}  with $\ell=0$ in the form $k_t=-\mathcal{E} a$. Letting $a=-D_x^{-1} F_j[k]$ with $F_0[k]=k_x$ produces the KdV hierarchy.
\end{proof}

We also obtain an interesting evolution equation for the angle $\phi$; the resulting flow may be thought of as the evolution of a mapping into $\RP^1$ obtained by composition of
$\gamma$ with projectivization.
\begin{prop} Equation~\eqref{Pinkall} induces the following evolution for $\phi$:
\begin{equation}\label{schwarz2}
\phi_t = \left[\Schwa(\phi) + \tfrac{1}{2}(\phi_x)^2\right] \phi_x.
\end{equation}
\end{prop}
\begin{proof}
Letting $\gamma_1=\cos \phi$ and $\gamma_2=\sin \phi$ in~\eqref{adownflow}, we compute $\phi_t=a\phi_x$. The claim follows from setting $a=-k=\Schwa(\phi) + \frac{1}{2}(\phi_x)^2$.
\end{proof}

\begin{rem} The expression $\Schwa(\phi) + \frac{1}{2} (\phi_x)^2$ is a modified Schwarzian derivative, originally introduced for diffeomorphisms of the circle, and it is invariant under the action of $PSL(2,\R)$ (see, e.g.~\cite{Seg81, Ovs01}). Remarkably, one can directly verify that, if $\phi_t= \left[\Schwa(\phi) + \lambda (\phi_x)^2\right]\phi_x$ for an arbitrary constant $\lambda$, the quantity
\[
u = -\left[\Schwa(\phi) + \lambda  (\phi_x)^2\right]
\]
is a solution of the KdV equation.
\end{rem}

\begin{rem}\label{pinkrem}
Equation~\eqref{Pinkall}  is related to Pinkall's flow for curves in the centroaffine plane~\cite{P95}, described above in Example \ref{pinkex}.
Let  $\Omega: I \rightarrow \R^2$ be  a star-shaped curve satisfying $\det(\Omega, \Omega_x)=1$.
Suppose we represent the curve in polar coordinates, by setting $\Omega=\alpha (\cos \theta, \sin \theta)$ for functions
 $\alpha>0$ and $\theta$ of $x$.  (Note that the latter must be an increasing function.)
Computing $\det(\Omega,  \Omega_x)=\alpha^2 \theta_x$ shows $\alpha=\theta_x^{-1/2}$, and it follows that
\[
\Omega_{xx}=\left[ \alpha_{xx}\alpha^{-1} -(\theta_x)^2\right] \Omega=-\left[ \tfrac{1}{2} \Schwa(\theta) + (\theta_x)^2 \right]\Omega,
\]
giving $p=\tfrac{1}{2} \Schwa(\theta) + (\theta_x)^2$ for the curvature.  If we write Pinkall's flow as $\Omega_\tau=\tfrac12 p_x \Omega -p\Omega_x$ for time variable $\tau$, then it follows that $\theta$ evolves by
\begin{equation}\label{CAthetaflow}
\theta_\tau=\left[ \frac{1}{2} \Schwa(\theta) + (\theta_x)^2\right] \theta_x.
\end{equation}
This in turn transforms to~\eqref{schwarz2} via the variable changes $\phi=2\theta$ and $t=\tfrac12 \tau$. Thus, we can interpret $\phi=2\theta$ as a defining a double cover that connects the geometric realization of the KdV hierarchy for sextactic L-curves to that for star-shaped curves in the centro-affine plane.
\end{rem}

\subsection{Normal indicatrix of a sextactic L-curve}

For sextactic L-curves, the normal indicatrix (traced out by frame vector $N$) is also an L-curve since, from~\eqref{firstLfrenet} with $\ell=0$, we have $\langle N, N \rangle=\langle N, N_x \rangle=0$.
To compute its curvature, we construct a normalized framing by setting $e_0=cN$, for $c$ a real-valued function, and computing
\[
e_1=e_{0}'=c_xN-\ri ckT.
\]
(Once more, for sake of readability, we use a prime when differentiating frame vectors with respect to $x$.)
By requiring that $\langle e_1, e_1 \rangle=1$ we obtain $c=k^{-1}$, giving $e_0=k^{-1} N$  and $e_1=-k^{-2}k_x N -\ri T$. Since $e_{1}'=\ri e_2 + k_N e_0$ (where $k_N$ denotes the curvature of the normal indicatrix), we use the normalization conditions~\eqref{innerproductconds} for the frame vectors to compute
\begin{equation}
\label{CurvNormIndic}
k_N=-\frac{1}{2} \langle e_{1}', e_{1}' \rangle=k  -\frac{k_{xx}}{k} + \frac{3}{2}\left(\frac{k_x}{k}\right)^2.
\end{equation}
Note that~\eqref{CurvNormIndic} can be written as $k_N=k-\Schwa({\theta})$, by letting $\theta_x:=k$.
\begin{prop} The normal indicatrix of a sextactic L-curve is also sextactic.
\end{prop}
\begin{proof}
Differentiating both sides of $e_1=-k^{-2}k_x N -\ri T$ and comparing with $e_{1}'=\ri e_2 + k_N e_0$ where $k_N$ is as in~\eqref{CurvNormIndic}, we find that
$e_2=k^{-1}k_x T -\tfrac12 \ri k^{-3} (k_x)^2 N -k\Gamma$.
Note that the frame vectors computed so far satisfy $\det(e_0, e_1,e_2) =  -\ri$, so to satisfy the determinant condition \eqref{determinantcond} we should
multiply each vector $e_i$ by $e^{i\pi/6}$; however, this will not change the coefficients of the frame vectors in the expressions for the derivatives $e_i'$.
In particular, expressing $e_2'$ in terms of $e_0$ and $e_1$ gives $e_{2}'=-\ri k_N e_1$, showing that $\ell_N=0$.
\end{proof}


%

When a sextactic L-curve evolves by a geometric flow, it is of interest to compute the induced evolution of its normal indicatrix.
For example, the evolution of $e_0$ induced by~\eqref{hzeroflow} is
\begin{equation}\label{e0inducedflow}
\frac{d e_{0}}{dt}=-\left(a -k^{-1}a_{xx}\right)_x e_0+\left( a-k^{-1}a_{xx}\right)e_1.
\end{equation}
which in the case of Pinkall's flow (where $a=-k$) specializes to
\[
\frac{d e_{0}}{dt}=\left(k-k^{-1}k_{xx}\right)_x e_0-\left(k-k^{-1}k_{xx}\right)e_1
\]
It does not seem possible, in general, to express the evolution of the curvature $k_N$ induced by \eqref{e0inducedflow} purely in terms of $k_N$ and its derivatives.
We leave open the question of whether the indicatrix transformation can lead to a relation between integrable PDEs for specific choices of $a$.


\subsection{Arclength-parametrized L-curves and the Kaup-Kuperschmidt hierarchy}

\begin{defn} An L-curve is {\em generic} if it is free of sextactic points (i.e., $\ell$ is nonvanishing for any choice of framing satisfying the equations of Prop. \ref{firstLadapt}).  A generic parameterized L-curve can always be reparametrized so that it has a normalized lift for which
$\ell= 1$ identically.  We refer to such L-curves as {\em pseudoconformal arclength-parametrized} or simply {\em arclength-parametrized}.
\end{defn}

It follows from Example \ref{KKR} above that choosing $a=4k^2-k_{xx}$ and $h=9$ yields a flow that preserves the set of arclength-parametrized L-curves and induces the Kaup-Kuperschmidt (KK) evolution for the curvature $k$.   Note that, in general, in order for a flow to preserve the condition $\ell = 1$, equation \eqref{Lelflow} implies that the functions $a$ and $h$ defining the flow must satisfy the identity
\begin{equation}\label{firstcompat}
18 a_x =D\left( h^{(4)}  - 10 k h_{xx}  - 5 k_x h_x -2 (k_{xx} - 4 k^2) h \right) - 2(k_{xx} - 4 k^2) h_x.
\end{equation}

We will show below that functions $a,h$ can be chosen so as to satisfy this identity,
and to generate the entire KK hierarchy for the evolution of curvature.  (In what follows, we base our formulation of the hierarchy on \S2.20 in \cite{JPlist}.)
Recall that the KK hierarchy
take the form $u_t = K_j[u]$, where $K_j$ is a polynomial in $u$ and its $x$-derivatives,
generated by successively applying a recursion operator $\fR$ to one of two `seeds':
$$K_0[u]=u_x, \qquad K_1[u] =u'''' + 5 u u''' + \tfrac{25}2 u' u'' + 5 u^2 u'.$$
Note that the KK equation $u_t = K_1[u]$ agrees with the evolution for $k$ in Example \ref{KKR} under the change
of variable $u=-2k$.  Under this change of variable, the compatibility condition \eqref{firstcompat} for $a$ and $h$ becomes
\begin{equation}\label{nonstretchu}
18a_x = (u_{xx} + 2 u^2) h_x+ D\left( ( D^4  + 5 u D^2  + \tfrac52 u_x D + u_{xx} + 2 u^2)  h \right).
\end{equation}

The `even' flows arise by applying $\fR$ to $K_0$, the `odd' flows by applying it to $K_1$, and $\fR$ is given by
\begin{multline*}
\fR = D^6 + 6u D^4 + 18u' D^3 + (9u^2+\tfrac{49}2u'') D^2+(30 u' u + \tfrac{35}2 u''')D
\\ +\tfrac{13}2 u''''+ \tfrac{41}2 u''u + \tfrac{69}4 (u')^2 + 4u^3 + K_1[u] D^{-1} + \tfrac12 u' D^{-1} \circ (u'' + 2u^2), \end{multline*}
where $D$ again denotes the total derivative with respect to $x$.
Because the recursion operator must be applicable to all $K_j$, the presence of
antiderivative operators in the last two terms of $\fR$ means that, for every $j\ge 0$ there must be local functions $L_j[u]$ and $M_j[u]$ such that
$$K_j = D L_j,\qquad (u''+2u^2) K_j = D M_j.$$
In particular, \eqref{nonstretchu} is satisfied by $h=L_j$ and
\begin{equation}\label{anons}
a = \tfrac1{18}\left[ M_j + \left( D^4  + 5 u D^2  + \tfrac52 u_x D + u_{xx} + 2 u^2 \right) L_j \right].
\end{equation}
Hence, the components of the KK hierarchy allows us to construct an infinite
(double) sequence of geometric flows that preserve arclength.  While Examples~\ref{trans} and~\ref{KKR} show how to choose $a$ and $h$ so as to induce the lowest levels of the KK hierarchy,  the following result shows how to induce  higher-level members of  this hierarchy.

\begin{thm}\label{KKthm} Let the normalized lift of an arclength-parametrized L-curve evolve according to
$$\Gamma_t = \left(a' -\ri( \tfrac23 u h' + \tfrac16 h'')\right) \Gamma + (a+\tfrac12 \ri h') T + h N,$$
with $h=L_j[u]$ and $a$ given by \eqref{anons} for $j\ge 0$, where $u=-2k$.  Then $u$ satisfies
\begin{equation}\label{KKcombo}
u_t = \tfrac19 K_{j+2}[u] - 3K_j[u].
\end{equation}
\end{thm}
\begin{proof}
When we substitute $k=-\tfrac12 u$ and $\ell=1$ into \eqref{Lkflow} the general form for the evolution of $u$ is
$$u_t = a u' + 2a' u + 2a'''-3h'.$$
Next we substitute in for $a$ from \eqref{anons}, and substitute for any derivatives of $M_j$
using $D M_j = (u''+2u^2) h'$, $D^2 M_j = D( (u''+2u^2) h')$ and so on.  The result is
\begin{multline*}
u_t = \tfrac19 h^{(7)} + \tfrac23 u h^{(5)} + 2u' h''''+(u^2 + \tfrac{49}{18}u'') h'''
+ (\tfrac{10}3 u u' + \tfrac{35}{18}u''') h''
\\+ ( \tfrac49 u^3 + \tfrac{41}{18} u u'' + \tfrac{13}{18} u'''' + \
\tfrac{23}{12} (u')^2) h' + (\tfrac19 u^{(5)} + \tfrac59 u u''' + \tfrac{25}{18} u' u''
+ \tfrac59 u' u^2) h + \tfrac1{18} u' M_j-3h'.
\end{multline*}
Then one checks that this is exactly the same as what results from applying $\fR$ to
$\tfrac19 h'$, except for the term $-3h'$ on the end. Since $h'=K_j[u]$ we obtain \eqref{KKcombo}.
\end{proof}

Of course, it is well known that the KK hierarchy arises as specialization of the Boussinesq hierarchy. This is discussed in detail in the context of curve flows in centroaffine $\R^3$ in \cite{CIM}. The argument described in Section 5.3 of  that paper can be adapted to obtain a similar realization of the KK hierarchy for arclength-parametrized L-curves as a special case of the Boussinesq flows constructed in Theorem \ref{bouthm}.  We leave further details to the interested reader.

\section{Transverse Curves and their Deformations}\label{Tsection}
In this section we consider regular parametrized curves $\gam:I \to S^3$ whose tangents are everywhere transverse to the pseudoconformal contact planes; we will refer to these as transverse curves or {\em T-curves}.
Parallel to Prop. \ref{firstLadapt}, we begin with constructing an adapted moving frame:

\begin{prop}Any T-curve $\gam:I \to S^3$ has a framing $(\Gamma, B, V)$ such that
\begin{equation}\label{firstTadapt}
\dfrac{d\Gamma}{dx} = \ri k \Gamma + \nu V, \quad \dfrac{dB}{dx} = \ell \Gamma - 2\ri k B,
\quad \dfrac{dV}{dx} =  m \Gamma -\ri \ell B +\ri k V
\end{equation}
for real-valued functions $k, \ell, m$ and $\nu \ne 0$.
\end{prop}
\begin{proof} Let $(e_0, e_1, e_2)$ be any framing for $\gam$, and suppose $e_0' = a e_0 + b e_1 + c e_2$.  Differentiating $\langle e_0, e_0 \rangle =0$
shows that $c$ is real and the transverse condition implies that $c$ is non-vanishing.  By making a change of frame \eqref{CoF} with $-\ri \overline{\mu} = b/c$ we can absorb the $e_1$ term in $e_0'$, so that we may assume $b=0$.  Now it follows from differentiating $\langle e_0, e_1\rangle = 0$, $\langle e_2, e_0 \rangle = \ri$, $\langle e_2, e_2 \rangle = 0$ and $\det(e_0, e_1, e_2) = 1$ that the new frame vectors satisfy
\begin{align*}
e_0' &= a e_0 + c e_2,\\
e_1'&= \ell e_0 + (\overline{a} - a) e_1, \\
e_2' &= m e_0 -\ri \overline{\ell}e_1 - \overline{a} e_2
\end{align*}
for some functions $m,\ell$ with $m$ real-valued.  By adding $\realpart(a/c) e_0$ to $e_2$ we absorb the real part of $a$, so that we now have
\begin{align*}
e_0 '&= \ri k  e_0 + c e_2,\\
e_1'&= \ell e_0 -2\ri k e_1, \\
e_2' &= m e_0 -\ri \overline{\ell}e_1 + \ri k e_2
\end{align*}
for some real-valued function $k$.  Finally, by using the scaling $e_0 \to \lambda e_0$, $e_1 \to (\overline{\lambda}/\lambda) e_1$,
$e_2 \to (1/\overline{\lambda}) e_2$ we may arrange that $\ell$ is real and non-negative.  Now re-label the frame vectors
$(e_0, e_1, e_2)$ as $(\Gamma, B, V)$ and let $c=\nu$.
\end{proof}

Note that the remaining freedom to adjust the frame is scaling $\Gamma \to \lambda \Gamma$,
$V \to \lambda^{-1} V$ for $\lambda$ real and simultaneously multiplying all frame vectors by a cube root of unity.  Consequently, points where $\ell$ vanishes are geometrically meaningful for the curve $\gam$ in $S^3$.  In fact, it is easy to see that these are the points where
the curve has second-order contact with its complex tangent line (i.e., the projection into $S^3$ of the plane spanned by $\Gamma$ and $\Gamma_x$),
so it makes sense to call these {\em inflection points}.

The sign of $\nu$ in \eqref{firstTadapt} is also unchanged under scaling, and is determined by whether the velocity of $\gam$ is positively or negatively
oriented with respect to the contact planes.  By reversing the sign of $x$ if necessary, we will assume from now on that $\gam'$ is positively oriented, and
use the scaling to arrange that $\nu=1$, i.e., the framing satisfies
\begin{equation}\label{secondTadapt}
\dfrac{d\Gamma}{dx} = \ri k \Gamma + V, \quad \dfrac{dB}{dx} = \ell \Gamma - 2\ri k B,
\quad \dfrac{dV}{dx} =  m \Gamma -\ri \ell B +\ri k V.
\end{equation}
As with L-curves, we will refer to this as a {\em normalized framing} for $\gam$.

\begin{prop} Let $\gam(x,t)$ be a smooth variation of T-curves, and let $(\Gamma, B, V)$ a smoothly-varying choice of normalized framing.   If we write
\begin{equation}\label{Tunitflow}
\Gamma_t = f \Gamma + g B +h V
\end{equation}
then $h$ is real-valued, $g=a+\ri b$ and $f = \ri v -\tfrac12 h_x$ for some real-valued functions $a,b,v$ satisfying
\begin{equation}\label{constraint1}
a'' + 3 k' b + 6 k b' + 3\ell (v- k h) - (m+9k^2) a = 0,
\end{equation}
where prime denotes $\partial/\partial x$.  Furthermore, the invariants $k,\ell,m$ evolve by the equations
\begin{subequations}\label{Tklmflow}
\begin{align}
k_t &= b \ell + v', \label{Tkflow} \\
\ell_t &= 3a k' + 6k a' + h \ell' + \tfrac32 \ell h' - b'' + b (m+ 9k^2),\label{Tellflow} \\
m_t &= a \ell' + 3 a' \ell + h m' + 2 h' m + 6b k\ell - \tfrac12 h'''.\label{Tmflow}
\end{align}
\end{subequations}
\end{prop}
\begin{proof} As in \eqref{hreal}, the reality of $h$ is necessary for $\Gamma$ to remain on the null cone.  Since $\nu=1$, we have
$\langle \Gamma_x, \Gamma\rangle = \ri$, and differentiating this with respect to $t$ gives $f+\overline{f} = -h_x$.  We obtain the relation
\eqref{constraint1} and the evolutions for invariants $k,\ell,m$ by a calculation that is parallel to that in the proof of Proposition \ref{Levolutions}.
\end{proof}

We wish to find choices of $a,b,h$ and $v$ that are local functions of the invariants $k,\ell,m$ and their derivatives which give
an integrable evolution equations for these invariants, while also satisfying the constraint \eqref{constraint1}.  We can easily satisfy
the latter equation if $\ell \ne 0$, for then it can be solved for $v$.  For the sake of simplicity, we will mostly concentrate on flows for curves that
have $\ell$ identically equal to a nonzero constant $\lambda$.  (If we think of the middle equation of \eqref{secondTadapt} as giving the evolution
of the `binormal' vector $B$, this condition is analogous to constant torsion for curves in Euclidean $\R^3$.)
The relation \eqref{constraint1} can then be rewritten as
\begin{equation}\label{constraint1spec}
\calM a + 3 \calS b + 3\lambda (v-kh)=0,
\end{equation}
where, for the sake of convenience, we define the linear differential operators
$$\calM = D^2 -(m+9k^2),\qquad \calS = k D + D \circ k.$$
However, setting $\ell = \lambda$ in \eqref{Tellflow} introduces an additional condition
\begin{equation}\label{constraint2}
3\calS a - \calM b + \tfrac32\lambda h_x = 0.
\end{equation}
The evolution equations for the remaining invariants are now
\begin{align*}
k_t &= \lambda b + v', \\
m_t &= \lambda (3a' + 6kb) + h m' + 2 h' m -\tfrac12 h'''.
\end{align*}

\begin{nex}\label{mikex} Setting $b=0$ and $a=1$ gives $\tfrac32 \lambda h_x = -3k_x$ in \eqref{constraint2}, so we may take $h=-\tfrac2{\lambda}k$.  Then solving
\eqref{constraint1spec} gives $v= \tfrac1{3\lambda} (m+3k^2)$, and the invariants evolve by
\begin{align*}
k_t &= \tfrac1{3\lambda} (m+3k^2)_x, \\
m_t &=\tfrac1{\lambda}(k''' -2k m' -4 k' m).
\end{align*}
(Since it can be absorbed by rescaling time, we may assume that $\lambda=1$.)
By the change of variables $u=k$ and $v=\tfrac13m + k^2$ and further re-scalings, this system is equivalent to equation (65) in \cite{MNW}.  In that paper, it is asserted that this system is known to be integrable, with a Lax pair  to be found in \cite{AC}.
\end{nex}

\begin{nex}\label{sinkex}  Recall from Theorem \ref{sexthm} above that commuting flows of the KdV hierarchy are recursively generated by $F_{j+1} = \calE D^{-1} F_j$
with $F_0=k_x$.  Thus, there are local functions $L_j = D^{-1} F_j$ of $k$ and its derivatives such that
$$L_{j+1} = D^{-1} \calE L_j = (D^2 - D^{-1} \calS ) L_j.$$
Thus, choosing $b=0$ and $a$ to be a constant multiple of $L_j$ ensures that we can satisfy the second constraint \eqref{constraint2}, which requires
$h=-\tfrac2{\lambda} D^{-1} \calS a$.
In this way we produce an infinite sequence of local evolution equations preserving $\ell = \lambda \ne 0$.
For example, setting $a=-L_0 = -k$ and $b=0$ leads to
$$k_t = \tfrac1{3\lambda}\left(k''-km\right)', \qquad m_t = -3\lambda k' + \tfrac3{\lambda}\left( 4mkk' + k^2 m' - k k''' - 3k' k''\right).$$
This system appears to have an infinite sequence of conserved densities, beginning with
\begin{multline*}
\rho_1 = k^2 - \tfrac19 m, \quad \rho_2 = k^2m + (k')^2, \\
\rho_3 = k^6 + \tfrac53 k^4 m - \tfrac5{27} k^2 m^2 - \tfrac1{729}m^3 + \tfrac53 k^2 (k')^2 - \tfrac59 m (k')^2 -\tfrac13 (k'')^2 - \tfrac{20}{27} k k' m' - \tfrac1{243} (m')^2 - \tfrac23 \lambda^2 k \rho_1, \ \ldots
\end{multline*}
Densities $\rho_i$ have been calculated up to $i=5$, and each is the unique conserved density (up to multiple) that is polynomial in $k,m$ and their derivatives and is homogeneous of weight $2i$.  (Here, we assign $k,m$ and $\lambda^2$ weights 1, 2 and 3 respectively, and each $x$-derivative increases weight by one.)   We do not know if these densities can be recursively generated, nor if the evolution equation systems produced by setting $a=-L_j$ for $j>0$
have similar sequences of conserved densities.
\end{nex}

\begin{nex}\label{lpreserving}
We may also satisfy the constraint \eqref{constraint1} by setting $v=kh$ and $a=b=0$, giving the following evolution equations for the invariants:
\begin{equation}\label{Tgzeroflows}
\begin{aligned}
k_t &= (kh)',\\
\ell_t &= h\ell' + \tfrac32 \ell h',\\
m_t &= hm' + 2mh' - \tfrac12 h'''.
\end{aligned}
\end{equation}
Since the last equation has the same form as equations in the KdV hierarchy (with $k$ replaced by $m$, and with an appropriate scaling of the variables), it is clear that there is a sequence of choices for $h$, as a function of $m$ and its derivatives, that induce evolution
equations in the KdV hierarchy for $m$ (we leave the details to the interested reader), while $k$ and $\ell$ evolve by homogeneous linear equations whose coefficients depending on $m$.

In particular, such curve evolutions would preserve the condition $\ell=0$, and we wish to dwell briefly on the geometric interpretation of that condition.

\begin{lemma} A T-curve with $\ell=0$ identically is pseudoconformally congruent to a mapping into $S^3$ whose image lies along a fiber of the Hopf fibration, when one uses affine coordinates to identify
$S^3$ with the unit sphere in $\C^2$ as in Lemma \ref{identifysphere}.
\end{lemma}
\begin{proof}
If $\ell=0$ identically in $x$ then the lift $\Gamma$ of the curve remains in a fixed complex plane through the origin in $\C^3$ (e.g., this plane is the orthogonal complement of the vector $B$, which by \eqref{secondTadapt} is fixed up to multiple).
This plane contains the linearly independent null vectors $\Gamma$ and $V$, so the Hermitian form $\langle \, , \rangle$ restricts to be nondegenerate with mixed signature on this plane.  Hence the intersection of the null cone $\scrN$ with this complex plane has real dimension three, and its image under projectivization is a circle. This circle contains the image of $\gam:I\to S^3$, since the latter lies in the
plane intersection of $S^3$ with a complex line in $\C^2$.  Although the pseudoconformal action does not preserve the Hopf fibration, we can use the group to arrange that the circle is a fiber of the Hopf fibration.

To see this, suppose that $B$ is a multiple of the vector $(0,0,1)$.  Then $\Gamma$ must be of the form
$(r e^{\ri \phi}, r e^{\ri \theta},0)$ for some functions $r,\phi,\theta$ of $x$ with $r >0$.  In terms of affine coordinates, the projection
into $\C^2$ has components $\gam_1 = \Gamma_1/\Gamma_0 = e^{\ri (\theta-\phi)}$ and $\gam_2=0$, and it follows that the $\C^2$-valued vector $\bgam$ satisfies
$\bgam_x = \ri(\theta_x - \phi_x) \bgam$.  Since this is an imaginary multiple of $\bgam$ itself, it is tangent to the Hopf fiber, and the image of $\gam$ remains on a single fiber for all $x$.
\end{proof}

Since these curves are congruent to maps into the circle formed by intersecting $S^3$ with the projectivization of the $z_2=0$ plane, and the subgroup of $SU(2,1)$ preserving this plane is $SU(1,1) \cong SL(2,\R)$, it is reasonable to expect that the remaining pseudoconformal invariant $m$ can be identified with the centroaffine invariants of Example~\ref{pinkex}.
\begin{prop}\label{samesame}
 Let the unit-speed lift of a curve $\gam$ with $\ell=0$ evolve by \eqref{Tunitflow} with $g=0$ and $f=-\tfrac12 h_x + \ri k h$.
Then there is a unit-modulus function $\mu(x,t)$ such that $\tGamma = \mu \Gamma$ remains in a real two-dimensional plane inside $\C^2$, satisfies $\det(\tGamma, \tGamma_x)=1$ as a vector in this plane, and evolves by
\begin{equation}\label{lzeroplanarflow}
\tGamma_t = -\tfrac12 h_x \tGamma + h \tGamma_x.
\end{equation}
In particular, when $h$ is chosen so that $m$ evolves by the KdV equation, then $\tGamma$ evolves by Pinkall's flow from Example \ref{pinkex}.
\end{prop}
\begin{proof}
By the previous lemma we may assume that $\Gamma$ takes value in $\C^2=(0,0,1)^\perp$ at $t=0$.  With $g=0$ we have $B_t = -\ell h -2\ri k h B$,
so since $\ell=0$ is preserved by \eqref{Tgzeroflows} we see that $B$ is fixed up to a unit-modulus multiple.  Hence $\Gamma$ remains in $\C^2$ for all $t$.

From \eqref{Tgzeroflows} the curvature $k$ evolves by $k_t = (kh)_x$, and so the linear differential equations
$$\mu_x = -\ri k \mu, \qquad \mu_t = -\ri k h \mu$$
are compatible and have a unit-modulus solution defined for all $x,t$.
Then the equations \eqref{secondTadapt} with $\ell=0$ imply that
\begin{equation}\label{tGammaFrenet}
\tGamma_x = \mu V, \quad (\mu V)_x = m\tGamma.
\end{equation}
Thus, $\det(\tGamma, \tGamma_x)$ is constant in $x$, and we may choose the initial value for $\mu$ to ensure that this determinant equals one.
From \eqref{Tunitflow} we compute that
$$\tGamma_t = \mu( (-\tfrac12 h_x  +\ri k h )\Gamma + h V) +\mu_t \Gamma = -\tfrac12 h_x \tGamma + h \tGamma_x.$$
In particular, the real plane spanned by $\tGamma$ and $\tGamma_x$ remains fixed for all $x$ and $t$.

It remains to identify $m$ with the curvature of $\tGamma$ as a map into the centroaffine plane.
Let $\tGamma = (r e^{\ri\phi}, r e^{\ri \theta}, 0)$, and note that the
conditions that $\tGamma$ is a null vector and the inner product with its $x$-derivative equals $\ri$ together imply that
$\phi_x = -\theta_x$ and $2r^2 \theta_x = 1$.  It follows that $\tGamma_{xx} = (r^{-1}r_{xx} - \theta_x^2) \tGamma$, and thus comparing with
\eqref{tGammaFrenet} shows that
$m = -\tfrac12\Schwa(\theta)-\theta_x^2$.  This is the same, up to an overall minus sign, as the centroaffine curvature of $\tGamma$ as computed in Remark \ref{pinkrem} above.
\end{proof}
\end{nex}

\section{Discussion and Open Questions}
Geometric realizations of integrable systems and hierarchies are not just interesting in their own right, but can also lead to new insights into the integrable structure of these PDE.  For example, in the proof of Theorem \ref{bouthm} we saw that the symplectic operator $\calP$ for the
Boussinesq hierarchy arises naturally when we compute how the free velocity components of a flow for L-curves determine the evolution
of the invariants of those curves.  So, it would be of interest to look for further connections between pseudoconformal geometry and integrability, including in the following areas:

\subsection*{AKNS-type Systems} For curve flows that induce an integrable system for the invariants, taking $x$- and $t$- derivatives of the framing yields a linear system of total differential equations whose compatibility condition is the underlying nonlinear PDE system.  For example, if we group the members $\Gamma, T, N$
of the normalized framing of an L-curve as columns in a matrix $F$ (as we did in the proof of Proposition~\ref{Lframevar}) then, when the curve flow realizes the Boussinesq system in the way described in Example \ref{bouex}, we have
$F_x = F U$ and $F_t = F V$
for $U$ as in \eqref{ULcurve} and
$$V=\begin{pmatrix} -\tfrac13\ri k & -\ell-\tfrac13\ri k_x & k^2-\tfrac13 k_{xx} \\ 0 & \tfrac23 \ri k& \tfrac13 k_x + \ri \ell \\
-1 & 0 & -\tfrac13 \ri k \end{pmatrix}.$$
Conversely, the Frenet equations $F_x = F U$ and $F_t = F V$ are compatible only if $k$ and $\ell$ satisfy the Boussinesq system \eqref{bsnq}.

Of course, the Boussinesq system also arises as the compatibility condition for its Lax pair (see, e.g., \S4.4 in \cite{CIM}), so it is natural to ask if there is a relation
between the solutions of these two linear systems.  In particular, can we interpolate a spectral parameter into the Frenet system?  One motivation here is
the analogous identification between the Frenet system associated the vortex filament flow (for so-called {\em natural frames}) and AKNS system for the NLS equation, which can be exploited to, for example, give a closure condition for the filament in terms of spectral data for the AKNS. {{We note that various approaches to introducing a spectral parameter for geometric flows have been proposed, and that the connections among the various approaches have generally not been explored. For example, the spectral parameter has been introduced through the normalization of an associated moving frame~\cite{FO99}, by extending the linear equation satisfied by the components of the curve to an eigenvalue problem~\cite{CIM}, or by identifying the spectral parameter with the reciprocal of the constant sectional curvature of the ambient space when the latter is a space form~\cite{DS94}.}}

\subsection*{Transformations} Another common feature of integrable systems are B\"acklund and Miura transformations, taking solutions of one equation to a one-parameter family of solutions to the same or another equation. Some of these transformations have their origins in geometry as, for example, the B\"acklund transformation for the sine-Gordon equation, which arises through the study of line congruences relating pairs of pseudospherical surfaces.  In the case of
curves in the pseudoconformal 3-sphere, each framing we construct leads to a secondary curve in the 3-sphere: the normal indicatrix for L-curves, or the tangent indicatrix traced out by projectivization the frame vector $V$ for T-curves.  It is natural to ask how the invariants of these indicatrices are related those of the primary curve, and furthermore whether, when the primary curve evolves by an integrable geometric flow, the invariants of the indicatrix evolve by a related integrable system.

In seeking geometric transformations, it is also worth mentioning that the pseudoconformal 3-sphere forms the boundary of the complex hyperbolic plane $\C H^2$, a complex manifold of constant negative holomorphic sectional curvature on which the same group $SU(2,1)$ acts as isometries.
One connection between curves in the boundary and objects in the interior is given by a `superposition formula' that associates a Hopf hypersurface in $\C H^2$
to a generic pair of L-curves on the boundary \cite{IveyHopf}.  Specifically, a Hopf hypersurface is one where applying the complex structure $J$ to the hypersurface normal
produces a principal direction on the hypersurface, and the two L-curves are traced out by the endpoints at infinity for the corresponding principal curves.  One might ask if this relationship could be exploited to produce a transformation between integrable flows for L-curves.

\subsection*{Further Examples}
Comparing Section \ref{Tsection} of this paper with previous sections will indicate that we have found relatively few examples of genuinely integrable flows for transverse curves, as compared to Legendrian curves.  (An exception is in the special case of curves which lie in a fiber of the Hopf fibration, but in that case  Proposition \ref{samesame} shows that the integrable flows there are essentially the same as the KdV flows in centroaffine geometry.)  Since transverse curves have a larger set of invariants, and multi-component integrable systems are rarer and less well-understood, it is thus important to identify further examples of integrable flows for this class of curves.

\bigskip
We hope to address some of the above questions in future work.

\section*{Acknowledgements} The authors thank Emilio Musso for helpful discussions on this project. The first author was partially supported by NSF grant DMS-1109017.

\end{document}